\documentclass[12pt,reqno]{amsart}
\usepackage{amsmath,amsfonts,amsthm,amssymb,color}
\usepackage[T1]{fontenc}
\usepackage{enumerate}
\usepackage{hyperref}
\hypersetup{ 
     colorlinks   = true,
     citecolor    = blue,
     linkcolor    = blue
}   
\usepackage{pdfsync} 
\usepackage[left=1in, right=1in, top=1.1in,bottom=1.1in]{geometry} 
\usepackage{xcolor}

\newcommand{\by}{\overline{Y}}

\newcommand{\der}{\delta}

\newcommand{\hb}{\hat{b}}

\newcommand{\iot}{\int_{0}^{t}}

\newcommand{\ott}{[0,\tau]}

\newcommand{\1}{{\bf 1}}

\newcommand{\R}{\mathbb R}

\newcommand{\be}{\mathbf{E}}

\newcommand{\bp}{\mathbf{P}}

\newcommand{\cac}{\mathcal C}

\newcommand{\ce}{\mathcal E}
\newcommand{\cf}{\mathcal F}
\newcommand{\ch}{\mathcal H}

\newcommand{\cn}{\mathcal N}

\newcommand{\al}{\alpha}

\newcommand{\ep}{\varepsilon}

\newcommand{\ga}{\gamma}
\newcommand{\gga}{\Gamma}

\newcommand{\la}{\lambda}

\newcommand{\oom}{\Omega}

\newcommand{\si}{\sigma}
\newcommand{\te}{\theta}
\newcommand{\tte}{\Theta}
\newcommand{\vp}{\varphi}

\newcommand{\lp}{\left(}
\newcommand{\rp}{\right)}
\newcommand{\lc}{\left[}
\newcommand{\rc}{\right]}
\newcommand{\lcl}{\left\{}
\newcommand{\rcl}{\right\}}
\newcommand{\lln}{\left|}
\newcommand{\rrn}{\right|}
\newcommand{\lla}{\left\langle}
\newcommand{\rra}{\right\rangle}

\newtheorem{theorem}{Theorem}[section]

\newtheorem{definition}[theorem]{Definition}

\newtheorem{hypothesis}[theorem]{Hypothesis}
\newtheorem{lemma}[theorem]{Lemma}
\newtheorem{notation}[theorem]{Notation}

\newtheorem{proposition}[theorem]{Proposition}

\theoremstyle{remark}
\newtheorem{remark}[theorem]{Remark}
\theoremstyle{remark}

\newcommand{\bean}{\begin{eqnarray*}}
\newcommand{\eean}{\end{eqnarray*}}
\newcommand{\ben}{\begin{enumerate}}
\newcommand{\een}{\end{enumerate}}
\newcommand{\beq}{\begin{equation}}
\newcommand{\eeq}{\end{equation}}

\begin{document}

\title[LAN property for fractional SDEs]
{LAN property for stochastic differential equations  \\ with additive fractional noise \\ and continuous time observation}

\author[Y. Liu]{Yanghui Liu}
\address[Yanghui Liu]{Department of Mathematics,
Purdue University,
150 N. University Street,
W. Lafayette, IN 47907,
USA.}
\email{liu2048@purdue.edu}
\urladdr{https://www.math.purdue.edu/~liu2048/}

\author[E. Nualart]{Eulalia Nualart}
\address[Eulalia Nualart]{Department of Economics and Business, Universitat Pompeu Fabra
and Barcelona Graduate School of Economics, Ram\'on Trias Fargas 25-27, 08005
Barcelona, Spain}
\email{eulalia.nualart@upf.edu}
\urladdr{https://www.upf.edu/web/eulalia-nualart}
\author[S. Tindel]{Samy Tindel}
\address[Samy Tindel]{Department of Mathematics,
Purdue University,
150 N. University Street,
W. Lafayette, IN 47907,
USA.}
\email{stindel@purdue.edu}
\urladdr{https://www.math.purdue.edu/~stindel/}

\date{}

\thanks{Second author acknowledges support from
the European Union programme FP7-PEOPLE-2012-CIG under grant agreement 333938. Third author  is supported by the NSF grant  DMS-1613163. The authors would like to thank David Nualart for fruitful discussions on the subject.}

\date{}

\begin{abstract}
We consider a stochastic differential equation with additive fractional noise with Hurst parameter $H>1/2$, and 
a non-linear drift depending on an unknown parameter. We show the Local Asymptotic Normality property (LAN)
of this parametric model with rate $\sqrt{\tau}$ as $\tau\rightarrow \infty$, when the solution is observed  continuously on the time interval $[0,\tau]$. The proof uses ergodic properties of the equation and a Girsanov-type transform. We analyse the particular case of the fractional Ornstein-Uhlenbeck
process and show that the Maximum Likelihood Estimator is asymptotically efficient in the sense of the Minimax Theorem.
\end{abstract}

\maketitle

\section{Introduction}
\label{sec:intro}

Let $B$ be a $d$-dimensional fractional Brownian motion (fBm) with Hurst parameter $H>1/2$. Let us recall  that $B$ is a centered Gaussian process defined on a complete probability space $(\oom,\cf,\bp)$. The law of $B$ is thus characterized by its covariance function, which is defined by
\begin{equation}\label{eq:cov-fbm}
R_{s;t}
\equiv
\be \big[  B_t^i \, B_s^j \big] = \frac 12 \lp |t|^{2H} + |s|^{2H} - |t-s|^{2H}  \rp \, \1_{ \{ 0\}}(i-j),
\qquad s,t\in\R.
\end{equation}
The variance of the   increments of $B$ is  then given by
\begin{equation} \label{var}
\be \lc | B_t^i - B_s^i |^2 \rc = |t-s|^{2H}, \qquad s,t\in\R, \quad i=1,\ldots, d,
\end{equation}
 and this implies that  almost surely the fBm paths are $\gamma$-H\"{o}lder continuous for any $\gamma<H$. 
 
In this article, we consider 
a time horizon $\tau\in (0, \infty)$ and
 the following  $\R^d$-valued stochastic differential equation driven by $B$:
\begin{equation}\label{eq:sde}
Y_t=y_0+\iot  b(Y_s;\te) \, ds +  \sum_{j=1}^d \si_j B_t^{j} , \qquad t\in\ott.
\end{equation}
Here $y_0 \in \mathbb{R}^d$ is a given initial condition, $B=(B^{1}, \ldots, B^d)$ is the aformentioned fractional Brownian motion, the unknown parameter 
$\te$ lies in a certain set $\tte$ which will be specified later on, $\{b(\cdot;\te), \, \te\in\tte \}$ is a known family of  drift coefficients with $b(\cdot;\te):\R^d\to\R^d $, and $\si_1, \ldots, \si_d \in \mathbb{R}^d$ are assumed to be  known diffusion coefficients.

There has been a wide interest in drift estimation for stochastic equations driven by fractional Brownian motion in the recent past, partly motivated by inverse problems in a biomedical context \cite{KS}. However, notice that the existing literature on the topic mainly focuses on the case where the dependence $\theta\mapsto b(x;\te)$ is linear and all coefficients are real-valued. In this situation least squares and maximum likelihood estimators (MLE) for the unknown parameter $\theta$ can be computed explicitly, and numerous results are available: the case of continuous observations of $Y$ and where the drift coefficient is linear in both $\theta$ and $x$ (that is, fractional Ornstein-Uhlenbeck type process) has been studied e.g. in \cite{BEO, HN-foup, KL, LB}, either in the ergodic ($\theta<0$) and non-ergodic ($\theta >0$) case. See also the monograph \cite{Rao} and the references therein. Results on parameter estimation based on discrete observations of $Y$ in the linear case can be found e.g. in \cite{BI, XZX}. The case of a dependence of the form $(\te, x)\mapsto \te \, b(x)$ is handled e.g. in \cite{KMM, TV}.

However, the case of a general multi-dimensional coefficient $b(\te,x)$ in equation \eqref{eq:sde} is also a very natural situation to consider, though we are only aware of the articles \cite{CT, NT} giving some positive answers for such a dependence. The current contribution has thus to be thought of as a step in this direction. Indeed, our main aim is to show that the model given by equation~\eqref{eq:sde} satisfies the Local Asymptotic Normality property (LAN) with rate $\sqrt{\tau}$ as $\tau\rightarrow \infty$, when the process $Y$ is observed continuously on the time interval $[0,\tau]$. 

Before we proceed to a specific statement of our results, let us describe the assumptions we shall work with, which are similar to the ones used in \cite{NT}. We start with a standard hypothesis on the parameter set $\tte$:
\begin{hypothesis}\label{hyp:Theta}
The set $\tte$ is compactly embedded in $\R^q$ for a given $q\ge 1$.
\end{hypothesis}

In order to describe the assumptions on our coefficients $b$, we will use the following notation for partial derivatives:
\begin{notation}
Let $f:\R^d\times\tte\to\R$ be a $\cac^{p_1,p_2}$ function for $p_1,p_2\ge 0$. Then for any tuple $(i_1,\ldots i_{p})\in\{1,\ldots,d\}^{p}$, we set $\partial^{i_1\ldots i_{p}}_x f$ for $\frac{\partial^{p} f}{\partial x_{i_1}\ldots \partial x_{i_{p}}}$. 
 Analogously, we use  
the notation $\partial^{i_1\ldots i_{p}}_{\theta} f$ for $\frac{\partial^{p} f}{\partial \theta_{i_1}\ldots \partial \theta_{i_{p}}}$,
where $(i_1,\ldots i_{p})$ is a tuple in $\{1,\ldots,q\}^{p}$. Moreover, we will write 
$\partial_x  f$ resp. $\partial_{\theta} f$ for the Jacobi-matrices
$(\partial_{x_1} f, \ldots, \partial_{x_d}f)$ and 
$(\partial_{\te_1} f, \ldots, \partial_{\te_q}f) $. Finally, for the sake of simplicity, we denote by  $\langle \cdot, \cdot \rangle$ the Euclidean scalar product in $\R^q$ or $\R^d$, and by $\vert \cdot \vert$ the corresponding Euclidean norm.
\end{notation}

Let us now state the linear growth plus inward assumptions on our drift coefficient ensuring ergodic properties for the process $Y$:
\begin{hypothesis}\label{hyp:f-coercive} 
\noindent
\emph{(i)} There exists $\alpha>0$ such that for every $x,y\in\R^{d}$ and $\te\in\tte$ we have
\begin{equation*}
\lla  b(x;\te)-b(y;\te), \, x-y\rra \le  - \al \,  |x-y|^2.
\end{equation*}

\noindent
\emph{(ii)} We have $b \in \mathcal{C}^{2,1} (\mathbb{R}^d \times \tte; \mathbb{R}^d)$, with $\partial_{x}b$ and $\partial_{xx}^{2}b$ uniformly bounded in $(x,\te)$. 

\noindent
\emph{(iii)} We have $\hat{b}:=\partial_{\theta} b \in \mathcal{C}^{2,0} (\mathbb{R}^d \times \tte; \mathbb{R}^d\times \mathbb{R}^q)$, with $\partial_{x}\hat{b}$ and $\partial_{xx}^{2}\hat{b}$ uniformly bounded in $(x,\te)$. 

\noindent
\emph{(iv)} There exists $c>0$ such that for every $x, y \in\R^{d}$ and $\te, \te_0\in\tte$, the following Lipschitz and growth conditions  hold:
\begin{equation*} 
\begin{split}
&|b(x;\te)| \le c \lp  1+|x|\rp,  \qquad   |\hat{b}(x;\te)| \le c \lp  1+|x| \rp,
\qquad
 \vert \hat{b}(x;\te)- \hat{b}(x;\te_0) \vert \leq c \vert \theta-\theta_0 \vert (1+\vert x \vert). 
\end{split}
\end{equation*}
\end{hypothesis}

Furthermore, we suppose that the coefficient $\si$ fulfills an invertibility condition of the following form:

\begin{hypothesis}\label{hyp:invertible-sigma}
The number $d$ of driving fBm equals the dimension of the state space $\R^{d}$ for $Y$. In addition, denoting by $\sigma$ the $d \times d$ matrix with columns
$\sigma_1,\ldots,\sigma_d$,  we assume that $\sigma$ is invertible. 
\end{hypothesis}

Notice that Hypothesis \ref{hyp:f-coercive} and \ref{hyp:invertible-sigma} entail the following result (see next section for more details): for a given $\theta \in \tte$  the solution to equation (\ref{eq:sde}) 
satisfies a.s. 
$ \lim_{t\to\infty} |Y_{t}-\by_{t}| =0$, where $\by=\{\by_{t}, t\geq 0\}$ is a   stationary and ergodic stochastic process.

Let us now introduce the concept of LAN property in our context. Towards this aim,  for any $\theta \in \Theta$ and $\lambda < H$, we denote by ${\bp}_{\theta}$  (resp. ${\bp}^{\tau}_{\theta}$)  the probability laws of the solution to equation (\ref{eq:sde})  in the spaces
$\mathcal{C}^{\lambda}(\R_+;\R^d)$  (resp. $\mathcal{C}^{\lambda}([0,\tau];\R^d)$). 
Moreover, we assume that the process $Y$ is observed continuously in $[0,\tau]$.
In this context, the definition of LAN property takes the following form:

\begin{definition}
We say that the parametric statistical model $\{{\bp}_{\theta}, \,\theta \in \Theta\}$ satisfies the LAN property at $\theta \in \Theta$ if
there exist:

\noindent
\emph{(i)} A $q\times q$ invertible matrix $\varphi_{\tau}(\theta)$,

\noindent
\emph{(ii)} 
A $q \times q$ positive definite matrix $\Sigma(\theta)$, 

\noindent
such that for any $u \in \R^q$, the following limit holds true as $\tau\rightarrow \infty$:
\begin{equation*} \begin{split}
\log \lp \frac{d {\bp}^{\tau}_{\theta+\varphi_{\tau}(\theta) u}}{d {\bp}^{\tau}_{\theta}} \rp
\,\overset{\mathcal{L}({\bp}_{\theta})}{-\!-\!\!\!\longrightarrow} \,
 u^{\rm T} \mathcal{N}(0,\Sigma(\theta)) -\frac{1}{2} u^{\rm T}\Sigma(\theta) u,
\end{split}
\end{equation*}
where $\mathcal{N}(0,\Sigma(\theta))$ is a centered $q$-dimensional Gaussian random variable with covariance matrix $\Sigma(\theta)$. 
\end{definition}

The LAN property is a fundamental concept in asymptotic theory of statistics, which was developed by Le Cam \cite{LC}. The essence of LAN is that the local log-likelihood ratio is asymptotically normally distributed, with a locally constant covariance matrix and a mean equal to minus one half the variance.
The main application of the LAN property is the following Minimax Theorem. It asserts (if LAN holds true) an asymptotic (minimax) lower bound for the risk with respect to a loss function, for any sequence $\hat{\theta}_{\tau}$ of estimators of $\theta$.  

A loss function is defined as a function $l:\mathbb{R}^q\to[0,+\infty)$ satisfying the following properties: 
\begin{itemize}
\item For any $u \in \R^q$, $\ell(u)=\ell(-u)$, $\ell$ is continuous at $0$, $\ell(0)=0$ but is not identically $0$.

\item For all $c>0$ the sets $\{u: \ell(u) <c\}$ are convex.

\item $\ell(u)$ has growth as $\vert u \vert \rightarrow \infty$ less that
$e^{\epsilon \vert u \vert^2}$, for any $\epsilon>0$.
\end{itemize}

\begin{theorem}[Minimax Theorem]  \textnormal{\cite[Theorem II.12.1]{IH}, \cite[Theorem 2.4]{KY}}\label{thm:minimax}
Suppose that the family of probability measures $(\bp_{\theta})_{\theta\in\Theta}$ satisfies the LAN property at a point $\theta$. Let $(\hat{\theta}_{\tau})_{\tau \geq 0}$ be a family of  estimators of the parameter $\theta$. Then for any loss function $\ell$ we have:
\begin{equation}\label{eq:low-bnd-minimax}
\lim_{\delta\to 0}\liminf_{\tau\to\infty}\sup_{\theta'\in \Theta:\vert\theta'-\theta\vert<\delta}\be_{\theta'}
\left[\ell\left(\varphi_{\tau}^{-1}(\theta)\left(\hat{\theta}_{\tau}-\theta'\right)\right)\right]\geq \be_{\theta}\left[\ell\left(Z\right)\right],
\end{equation}
where $\mathcal{L}(Z)=\mathcal{N}(0,\Sigma(\theta)^{-1})$.
\end{theorem}
Observe that when $\ell(u)=\vert u \vert^2$, the lower bound in \eqref{eq:low-bnd-minimax} is simply the trace of $\Sigma(\theta)^{-1}$. 
A sequence of estimators that attains this asymptotic bound for some loss function $\ell$  is  called asymptotically efficient for this loss function. Therefore, Theorem \ref{thm:minimax} opens the way to a theory which mimics the concept
of efficient estimator from the Cram\'er-Rao lower bound.

As one can see, the LAN property is an important tool in order to quantify the identifiability of a system.
The aim of this paper is thus to show the following result for our equation~\eqref{eq:sde}.
\begin{theorem} \label{lan}
Assume Hypothesis \ref{hyp:f-coercive}, and recall that ${\bp}_{\theta}^{\tau}$  stands for  the probability law of the solution to equation 
\textnormal{(\ref{eq:sde})}  in the space
$\mathcal{C}^{\lambda}([0,\tau];\R^d)$. Then, 
for any $\theta \in \Theta$ and $u \in \R^q$ fixed, as $\tau \rightarrow \infty$, we have:
\begin{equation} \label{limit}\begin{split}
\log \frac{d {\bp}^{\tau}_{\theta+\frac{u}{\sqrt{\tau}}}}{d {\bp}^{\tau}_{\theta}} 
\overset{\mathcal{L}({\bp}_{\theta})}{-\!\!-\!\!\!\longrightarrow} 
u^{\rm T} \mathcal{N}(0,\Sigma(\theta)) -\frac{1}{2} u^{\rm T}\Sigma(\theta) u,
\end{split}
\end{equation}
where the matrix $\Sigma(\theta)$ is defined by:
\begin{equation} \label{gamma_def}
\Sigma(\theta):=C_H \int_{\R_{+}^{2}} 
\frac{\be_{\theta}[(\hat{b}(\by_{0};\theta)-\hat{b}(\by_{r_{1}};\theta))^{\rm T} (\sigma^{-1})^{\rm T}\sigma^{-1}(\hat{b}(\by_{0};\theta)-\hat{b}(\by_{r_{2}};\theta))]}{r_{1}^{1/2+H} r_{2}^{1/2+H}} 
\, dr_{1} dr_{2},
\end{equation}
where the constant $C_{H}$ is defined by:
\begin{equation} \label{ch}
C_{H}= \frac{\sin^{2}\!\lp \pi\lp H-1/2 \rp \rp  \, \gga^{2}\lp H+1/2 \rp}
{2H\pi^{2} \, \sin(\pi H) \, \gga(2H)} 
=    [
{   \, \sin(\pi H) \, \gga(1/2-H)^{2} \gga(2H+1)}]^{-1} .
\end{equation}
In \eqref{gamma_def}, recall that $\by$ is the ergodic limit of our process $Y$ and that we have set $\hat{b}$ for the Jacobian matrix 
$\partial_{\theta} b$.
\end{theorem}

The LAN property for stochastic processes has been widely explored in the literature. To mention a few references close to our contribution, the case of ergodic diffusion processes observed continuously (with an unknown parameter on the drift coefficient) is dealt with in e.g. \cite[Proposition 2.2]{KY}. The proof follows from Girsanov's theorem, ergodic properties and the central limit theorem for Brownian martingales. The rate of convergence achieved in \cite[Proposition 2.2]{KY} is $\sqrt{\tau}$ like in our Theorem \ref{lan}, with a function $\gga(\te)$ given by:
$$
\Sigma(\theta)= \be_{\theta}[\hat{b}(\by;\theta)^{\rm T} \sigma^{-1}(\by)^{\rm T}\sigma^{-1}(\by)\hat{b}(\by;\theta)],
$$
which should be compared to our expression \eqref{gamma_def}.
 The case where the Brownian diffusion process is observed discretely was solved in \cite{G} using an integration by parts formula taken from the Malliavin calculus. Let us also mention that the LAN property is obtained in \cite{CGLL} for some discretely observed fractional noises including fractional Brownian motion. This last result is achieved by a direct expansion of log-likelihood functions, plus a thorough analysis based on properties of Toeplitz matrices.

Let us say a few words about the strategy we have followed in order to handle the case of equation \eqref{eq:sde}. In order to compare likelihood functions for different values of the parameter~$\theta$, we take the obvious choice of applying Girsanov's theorem for fractional Brownian motion (following the steps of \cite{MN}). We are then left with two technical difficulties in order to get asymptotic results: (i) Handle the singularities popping out of the fractional derivatives in the Girsanov exponent. (ii) Get ergodic results in H\"older type norms for our process $Y$. Let us also mention that CLTs for martingales, which were an essential tool in the diffusion case, are unavailable in our fBm situation.  

Interestingly enough, our general Theorem \ref{lan} applies to the fractional Ornstein-Uhlenbeck case (that is, $d=1$ and $b(x; \theta)=-\theta x$ for $\theta>0$.) In this context we find (i) That $\Sigma(\theta)$ does not depend on $H$. (ii) That the MLE reaches the lower bound in \eqref{eq:low-bnd-minimax}, and is thus asymptotically minimax efficient. 

Finally, the generalization of Theorem \ref{lan} to a fBm $B$ with Hurst parameter $H<1/2$ is   obviously a natural problem to consider. To this aim and following the strategy we have just summarized above, one can try to apply Girsanov's transform and ergodic theorems for the solution process $Y$ to~\eqref{eq:sde}. Then the main difficulty to implement this strategy stems from the fact that one is always left with evaluations of integrals similar to the right-hand side of relation \eqref{gamma_def}. Indeed, when $H<1/2$ it is easily seen that this kind of integral exhibits some divergences as $r_{1}$, $r_{2} \to \infty$, and we haven't found a proper way to cope with this additional difficulty. The LAN property for equation \eqref{eq:sde}  with $H<1/2$ is thus still an open and challenging question for us. 

Our paper is structured as follows: Section \ref{sec:preliminaries} is devoted to necessary preliminary results. Then we prove Theorem \ref{lan} in Section \ref{sec:proof-main-thm}, following the strategy described above. Eventually, in Section 4, we analyze the particular case of the fractional Ornstein-Uhlenbeck process. 

Let us close this introduction by giving a set of notations which will prevail until the end of the paper.

\begin{notation}\label{not:general}
We use the following conventions: for 2 quantities $a$ and $b$, we write $a \lesssim b$ if there exists a universal constant $c$ such that $a\le c\,b$. In the same way, we write $a\asymp b$ whenever $a \lesssim b$ and $b \lesssim a$. If $f$ is a vector-valued function defined on an interval $[0,\tau]$ and $s,t\in [0,\tau]$,  $\der f_{st}$ designates the increment $f_{t}-f_{s}$.
\end{notation}

\section{Auxiliary Results}\label{sec:preliminaries}

In this section we first recall some basic facts about stochastic analysis for a fBm $B$, and also about Young integrals. Then we shall derive some pathwise and probabilistic estimates for equation \eqref{eq:sde} under the coercive Hypothesis \ref{hyp:f-coercive}.

\subsection{Stochastic analysis related to B}
\label{sec:cameron-martin}

The ergodic properties of equation \eqref{eq:sde} are accurately encoded by the fBm representation given in \cite{Ha}, which goes back to the original paper~\cite{MVn}. We present this construction here for a one-dimensional fBm, for sake of simplicity. Obvious generalizations to the $d$-dimensional case are left to the reader. 

Let $W$ be a two-sided Brownian motion, and $H$ a Hurst parameter in $(0,1)$. We consider the process $B$ defined for $t\in\R$ by:
\begin{eqnarray}\label{eq:mandelbrot-rep-B}
B_{t} &=& c_{1}(H) \int_{\R_{-}} (-r)^{H-1/2} \lc dW_{t+r} - dW_{r} \rc \\
&=& 
c_{1}(H) \lcl  \int_{-\infty}^{0} \lc (t-r)^{H-1/2} -  (-r)^{H-1/2} \rc dW_{r} 
- \int_{0}^{t}  (t-r)^{H-1/2}  dW_{r} \rcl,
\notag
\end{eqnarray}
 where the constant $c_{1}(H)$ is defined by:
\begin{equation}\label{eq:def-c1(H)}
c_{1}(H)
=  
\frac{\lc 2H \, \sin(\pi H) \, \gga(2H)\rc^{\frac12}}{\gga\lp H+1/2 \rp}.
\end{equation}
Then $B$ is well-defined as a fBm, that is a centered Gaussian process with covariance given by relation \eqref{eq:cov-fbm}. Equation \eqref{eq:mandelbrot-rep-B} is often referred to as Mandelbrot's representation of fractional Brownian motion.

The process $B$ defined as \eqref{eq:mandelbrot-rep-B} is closely related to the following fractional derivatives: for $\al\in(0,1)$ and $\vp\in\cac_{c}^{\infty}(\R)$ we set
\begin{equation}\label{eq:def-frac-deriv-intg}
\lc D_{+}^{\al}\vp \rc_{t} =\frac{\alpha}{\Gamma(1-\alpha)} \int_{\R_{+}} \frac{\vp_{t}-\vp_{t-r}}{r^{1+\al}} \, dr,
\quad\text{and}\quad
\lc I_{+}^{\al}\vp \rc_{t} = \frac{1}{\Gamma(\alpha)}\int_{\R_{+}} \vp_{t-r} \, r^{\al-1} \, dr.
\end{equation}
With this notation in mind, the following proposition identifies a convenient operator which transform $W$ into $B$ as in \eqref{eq:mandelbrot-rep-B}. We refer to \cite[Lemma 3.6]{Ha} for further details.

\begin{proposition}\label{prop:K-H}
For $w\in \cac_{c}^{\infty}(\R)$ such that $w_{0}=0$ and $H\in(0,1)$, set 
\begin{equation*}
\lc K_{H} w \rc_{t}
=
c_{1}(H) \int_{\R_{-}} (-r)^{H-1/2} \lc \dot{w}_{t+r} - \dot{w}_{r} \rc dr .
\end{equation*}
Then the following holds true:
 
\noindent
\emph{(i)}
The operator $K_{H}$ admits the following expression:
\begin{equation*}
\lc K_{H} w \rc_{t}=
\begin{cases}
 c_{2}(H) \lp [ I_{+}^{H-1/2}w ]_{t} -  [I_{+}^{H-1/2}w ]_{0} \rp, & \text{for } H>\frac12  \\
-c_{2}(H) \lp [ D_{+}^{1/2-H}w ]_{t} -  [D_{+}^{1/2-H}w ]_{0} \rp, & \text{for } H<\frac12 ,
\end{cases}
\end{equation*}
where the constant $c_{2}(H)$ is given by $c_{2}(H)=c_{1}(H) \Gamma(H+\frac12)$, with $c_{1}(H)$ defined by \eqref{eq:def-c1(H)}.

\noindent
\emph{(ii)} Define for every $w\in C^{\infty}_{c}(\R )$ the norm 
\begin{eqnarray*}
\|w\|_{H} &=& \sup_{s,t\in \R} \frac{|w(t) - w(s)|}{|t-s|^{(1-H)/2} ( 1+|t|+|s| )^{1/2} }.
\end{eqnarray*}
Define the Banach space $\ch_{H}$ to be the closure of $C^{\infty}_{c} (\R)$ under the norm $\|\cdot\|_{H}$. 
Then $K_{H}$ can be extended continuously to $\ch_{H}$.

\noindent
\emph{(iii)}
For all $H\in(0,1)$ we have  $K_{H}^{-1}=-(c_{2}(H)c_{2}(1-H))^{-1} K_{1-H}$. 
\end{proposition}

Also notice that in the sequel,   our underlying Wiener  space   will be defined asd   $(\ch_{H}, \bp)$, where $\ch_{H}$ is introduced in Proposition \ref{prop:K-H} (ii), and $\bp$ is the unique Wiener measure on $\ch_{H}$ (see \cite{Ha} for further details).  The law $\bp^{\tau}_{\theta}$ of the solution process $Y|_{[0,\tau]}$ is defined as the image of $\bp$ by the map $\omega \mapsto Y(\omega)|_{[0,\tau]}$.

\subsubsection{Girsanov transform}

Let us set the ground for our Girsanov transform by relating a general drift with respect to a fractional Brownian motion with its counterpart with respect to $W$ (we continue in the one-dimensional setting for the sake of simplicity). This proposition is a slight extension of \cite[Lemma 4.2]{Ha}. 

\begin{proposition}
Let $H\in(1/2,1)$, $b^{1},b^{2}$ be two paths in $I_{+}^{H-1/2}(L^{2}(\R))$, and suppose that $b^{j}=K_{H}w^{j}$ for $j=1,2$, where $K_{H}$ is defined at Proposition \ref{prop:K-H}. We also assume that for $t\ge 0$, $b^{1}$ and $b^{2}$ are linked by the relation:
\begin{equation}\label{eq:rel-b1-b2}
b^{2}_{0} = b^{1}_{0}\,,
\quad\text{and}\quad
b^{2}_{t} = b^{1}_{t} + \int_{0}^{t} g_{b,s} \, ds\,,\quad t\geq 0
\end{equation}
where   $g_{b}$ is a function in $I^{H-1/2}_{+}(L^{2}(\R))$ and $g_{b}=0$ for $t<0$. Then we also have:
\begin{equation*}
w^{2}_{t} = w^{1}_{t} +  \int_{0}^{t} g_{w,s} \, ds\,,
\end{equation*}
with 
\begin{eqnarray*}
 g_{w} = 
  c_{2}(H)^{-1} \, D_{+}^{H-1/2} g_{b}\,,  
\end{eqnarray*}
where we recall   that the constant $c_{2}(H)=c_{1}(H) \Gamma(H+\frac12)$ has been introduced in Proposition~\ref{prop:K-H}.
\end{proposition}

\begin{proof}
Let us write $G_{b,t}=\int_{0}^{t} g_{b,s} \, ds$ for $t\geq 0$ and $G_{b,t}=0$ for $t<0$, so that \eqref{eq:rel-b1-b2} can be read as:
\begin{equation}\label{b}
b^{2}_{t} = b^{1}_{t} + G_{b,t}\,.
\end{equation}
Moreover, invoking the fact that $b^{j}=K_{H}w^{j}$ for $j=1,2$, one can recast \eqref{b} into:
\begin{equation*}
K_{H} w^{2} = K_{H} \lp w^{1}  + G_{w}  \rp,
\quad\text{where}\quad
G_{w} = K_{H}^{-1} G_{b}\,.
\end{equation*}
Therefore, thanks to the relation $K_{H}^{-1}=-c_{2}(H)^{-1}c_{2}(1-H)^{-1} K_{1-H}$, we end up with:
\begin{equation}\label{eqn.gw}
g_{w}
=
-c_{2}(H)^{-1}c_{2}(1-H)^{-1} D_{+}^{1} K_{1-H} G_{b }\,.
\end{equation}
 Since  $G_{b} \in I^{H-1/2}(L_{2})$, the operator $K_{1-H}$ 
 can be interpreted as a fractional differentiation. Specifically, applying the definition of $K_{1-H}$ contained in        Proposition \ref{prop:K-H} (i), we have 
 \begin{eqnarray}\label{eqn.KI}
K_{1-H} G_{b} =-c_{2}(1-H)( (D^{H-1/2} G_{b})_{t}- (D^{H-1/2} G_{b} )_{0}  ).
\end{eqnarray}
 Substituting \eqref{eqn.KI} into \eqref{eqn.gw} we obtain
\begin{equation*}
g_{w}
=
c_{2}(H)^{-1} D_{+}^{1} D_{+}^{H-1/2} G_{b}
=
c_{2}(H)^{-1} D_{+}^{H-1/2} g_{b}\,,
\end{equation*}
which is our claim. 
\end{proof}

The previous proposition yields a version of Girsanov's theorem in a fractional Brownian context, inspired by \cite{MN} (see also \cite{DPT,NO02}):
\begin{proposition} \label{girsanov}
Let $B$ be a fractional Brownian motion with   Hurst parameter $H \in (0,1)$ and let $g$ be an adapted process in $I_{+}^{H-1/2}(L^{2}(\R))$. We set $V_{t}=\int_{0}^{t} g_{s} \, ds$ and $Q_t= B_t +V_t$ for $t\ge 0$. Then $Q$ is a fractional Brownian motion on $[0,\tau]$, under the probability $\tilde \bp$ defined by $\frac{d\tilde \bp}{d\bp}_{\mid_{[0,\tau]}}= e^{-L}$, where  \begin{equation}\label{eq:def-L-girsanov-1d}
L=\frac{1}{c_{2}(H)} \, \int_{0}^{\tau} [ D_{+}^{H-1/2}g]_u \, dW_u
+ \frac{1}{2 (c_{2}(H))^{2}} \int_{0}^{\tau}  [ D_{+}^{H-1/2}g]_u^2 \, du ,
\end{equation}
provided $D_{+}^{H-1/2}g$   satisfies Novikov's conditions on $[0,\tau]$.
\end{proposition}

\subsection{Generalized Riemann-Stieltjes Integrals}\label{RS_integrals}

We are focusing here on the case of a Hurst parameter $H>1/2$ for sake of simplicity, so that all the stochastic integrals with respect to $B$ should be understood in the Young (or Riemann-Stieltjes) sense. In order to recall the definition of this integral, we set
\begin{equation}\label{eq:def-holder-norms}
| f |_{\infty;[a,b]}= \sup_{t \in [a,b]} |f(t)|, \qquad |f |_{\lambda;[a,b]} = \sup_{s,t \in [a,b]} \frac{|f(t)-f(s)|}{|t-s|^{\lambda}} ,
\end{equation}
where $f: \R \rightarrow \R^n$ and $\lambda \in (0,1)$. We also recall the classical definition of H\"older-continuous functions:
\begin{equation*}
C^{\lambda}([a,b]; \R^{n})
=
\lcl f:[a,b]\to\R^{n} ; \,  | f |_{\infty;[a,b]} + |f |_{\lambda;[a,b]} < \infty \rcl .
\end{equation*}

Now, let $f \in C^{\lambda}([a,b]; \R)$ and
$g \in C^{\mu}([a,b]; \R)$ with $\lambda + \mu >1$.  Then it is well known that the Riemann-Stieltjes integral $\int_{a}^{b} f_{s} \, dg_{s}$  exists, and can be expressed as a limit of Riemann sums. 

In the sequel, the only property on Young's integral we shall use is the classical chain rule for changes of variables. Indeed, let $f \in C^{\lambda}([a,b];\R)$ with
 $\lambda >1/2$ and $F \in C^{1}(\R;\R)$.  Then we have:
\begin{align}{\label{ito_proto}} 
 & F(f_{y})  -F(f_{a}) =   \int_{a}^{y} F'(f_{s}) \, df_{s}, \qquad y \in [a,b].
\end{align}  

\subsection{Basic properties of solutions to SDEs}
In order to deal with stationary solutions, let us first extend our equation as a process indexed by $(-\infty,\tau]$, by considering a process $Y$ solution to:
\begin{equation}\label{eq:sde-on-R}
dY_t=  b(Y_s;\te) \, dt +  \sum_{j=1}^d \si_j  dB_t^{j} , \quad t\in\ott,
\quad\text{and}\quad
Y_{t} = y_{0}, \quad t\in(-\infty,0].
\end{equation}
We first state the existence and uniqueness result for equation (\ref{eq:sde-on-R}) borrowed from 
\cite[Lemma 3.9]{Ha}.
\begin{proposition}\label{prop:exist-unique}
Under Hypothesis \ref{hyp:f-coercive}, there exists a unique continuous pathwise solution 
to equation \textnormal{(\ref{eq:sde-on-R})} on any arbitrary interval $[0,\tau]$. Moreover the map $Y: (y_0, B) \in \R^d \times \mathcal{C}([0,\tau]; \R^d) \rightarrow \mathcal{C}([0,\tau]; \R^d)$ is locally Lipschitz continuous.
\end{proposition}

In addition, exploiting the integrability properties of the stationary fractional
Ornstein-Uhlenbeck process, the following uniform estimates hold true. They are shown with more details in \cite{GKN}, and are obtained by comparing $Y$ with the solution to the equation $dU_{t}=-U_{t} \, dt + dB_{t}$.

\begin{proposition}\label{prop:bnd-moments-Y}
Assume Hypothesis \ref{hyp:f-coercive} holds true and let $Y$ be the solution to equation~\eqref{eq:sde-on-R}. Then for any $\te\in\tte$ and $p\ge 1$  there exist constants $c_p, k_p >0$ such that 
\begin{equation*}
\be \lc |Y_t|^{p} \rc \le c_p, \qquad\mbox{and}\quad 
\be \lc |\delta Y_{st}|^{p} \rc \le k_p |t-s|^{pH}, 
\end{equation*}
for all $s,t\ge 0$, where we recall that we have set $\delta Y_{st}=Y_{t}-Y_{s}$ as mentioned in Notation \ref{not:general}.
\end{proposition}

Let us now state a path-wise estimates on the increments of $Y$,
where we translate the moment estimates of Proposition \ref{prop:bnd-moments-Y} into path-wise bounds along the same lines as in   \cite{KN}.
\begin{proposition}\label{prop:pathwise-incr-bnd-Y}
Assume Hypothesis \ref{hyp:f-coercive} holds true and let $Y$ be the solution to equation~\eqref{eq:sde-on-R}. Then for all $\ep \in (0,H)$ there exists a random variable $Z_{\ep} \in\cap_{p\ge 1}L^{p}(\oom)$ such that almost surely we have:
\begin{equation}\label{eq:pathwise-incr-bnd-Y}
|Y_{t}| \le Z_{\ep} \lp 1 + t \rp^{2\ep},
\quad\mbox{and}\quad 
|\delta Y_{st}| \le Z_{\ep} \lp 1 + t \rp^{2\ep}  |t-s|^{H-\ep} ,
\end{equation}
uniformly for $0\le s \le t$. 
\end{proposition}

\begin{proof}
We focus on the bound for $\delta Y_{st}$, the estimate on  $|Y_{t}|$ being shown in the same manner.
Set $\ga=H-\ep$ and consider $n\ge 1$. Thanks to Garsia's lemma (see e.g. \cite{grr}) and recalling our Notation \ref{not:general}, for all $s,t\in[0,n]$ we have $|\delta Y_{st}|\lesssim \cn_{\ga,p,n} |t-s|^{\ga}$, where $p\ge 1$ and $\cn_{\ga,p,n}$ is defined by:
\begin{equation*}
\cn_{\ga,p,n} \equiv
\left( \int_0^n\int_0^n \frac{|\delta Y_{uv}|^{p}}{|v-u|^{\ga p+2}} \, du dv \right)^{\frac{1}{p}}.
\end{equation*}
Furthermore, owing to Proposition \ref{prop:bnd-moments-Y} we immediately get:
\begin{equation*}
\be^{1/p}\lc \cn_{\ga,p,n}^{p}  \rc \lesssim n^{\ep},
\end{equation*}
under the constraint $p\ge\ep^{-1}$. Set now $Z_{\ep}=\sup_{n \in \mathbb{N}} \frac{\vert \cn_{\ga,p,n} \vert}{n^{2 \ep}}$. Then for all $q>\ep^{-1}$ we have:
\begin{equation} \label{hol}
\be\lc Z_{\ep}^{q} \rc
=
\be \lc \sup_{n \in \mathbb{N}} \frac{\vert \cn_{\ga,p,n} \vert^{q} }{n^{2 \ep q}} \rc
\leq \sum_{n=1}^{\infty}
\frac{\be \lc\vert \cn_{\ga,p,n}\vert^q \rc}{n^{2 \ep q}} \lesssim \sum_{n=1}^{\infty}\frac{1}{n^{\ep q}}< \infty,
\end{equation}
namely we have found $Z_{\ep} \in\cap_{p> \ep^{-1}}L^{p}(\oom)$. Obviously this also yields $Z_{\ep} \in\cap_{p\ge 1}L^{p}(\oom)$.
Finally, replace $n \geq 1$ by $[t]+2$ to obtain our pathwise bound \eqref{eq:pathwise-incr-bnd-Y}.
\end{proof}

We shall need the following bound for the solution. This will be invoked while checking Novikov's condition.
\begin{lemma}
\label{bounds}
Assume Hypothesis \ref{hyp:f-coercive} and let $Y$ be the solution to \eqref{eq:sde-on-R}. Then, for all $t \geq 0$, $Y_{t}$ satisfies the following bound:
\begin{eqnarray}\label{eqn.bdY}
|Y_t| \leq \kappa_{1} e^{\kappa_{2} t} \sup_{s \in [0,t]} \vert B_s \vert, 
\end{eqnarray}
where $\kappa_{1}$ and $\kappa_{2}$ are two positive constants depending on $b$ and $\si$.  Furthermore, the increments of $Y$ are such that for all $r<s$:
\begin{eqnarray}\label{eqn.bddY}
|\delta Y_{r s}| \leq \kappa_{3} (s-r) ( 1+ e^{\kappa_{4} s} \sup_{u\in [0, s]} |B_{u}| ) + \kappa_{5} |\delta B_{rs}|.
\end{eqnarray}

\end{lemma}

\begin{proof}
Using \eqref{eq:sde-on-R} and the linear growth of $b$, we have
$$
\vert Y_t \vert \leq \vert y_0 \vert+ct+\Vert \sigma \Vert \vert B_t \vert +c\int_0^t
\vert Y_s \vert ds. 
$$
 Gronwall's lemma concludes the proof of \eqref{eqn.bdY}. In order to prove \eqref{eqn.bddY}, we just write 
\begin{eqnarray}\label{eqn.dY}
\delta Y_{rs} &=& \int_{r}^{s} b(Y_{u}; \theta) du + \sum_{j=1}^{d} \si_{j} \delta B_{rs}.
\end{eqnarray}
According to Hypothesis \ref{hyp:f-coercive} (iv) and \eqref{eqn.bdY} we have 
\begin{eqnarray*}
|b(Y_{u}; \theta)| \leq c_{b} (1+|Y_{u}|)
\leq c_{b} (1+\kappa_{1} e^{\kappa_{2}u} \sup_{v\leq u} |B_{v}|).
\end{eqnarray*}
Plugging this bound into \eqref{eqn.dY} we easily get \eqref{eqn.bddY}. 
\end{proof}

\subsection{Ergodic Properties of the SDE} 

We recall here some basic facts about the limiting behavior of equation \eqref{eq:sde-on-R}, mainly taken from \cite{GKN}. We still work on the Wiener space $(\ch_{H}, \bp)$ introduced in Section \ref{sec:cameron-martin}, and seen as the canonical probability space.
Together  with the shift operators $\theta_t: \Omega \rightarrow \Omega$ defined by
$$ \theta_t\omega(\cdot)= \omega(\cdot +t)- \omega(t), \qquad t \in \mathbb{R}, \quad \omega \in \Omega,$$
our probability space defines an  ergodic metric dynamical system,
see e.g. \cite{GA_S}. In
particular, the measure $\mathbf{P}$ is invariant to the shift
operators $\theta_t$, i.e. the shifted process $(B_{s}(\theta_t
\cdot))_{s \in \mathbb{R}}$ is still an $m$-dimensional fractional
Brownian motion and for any integrable random variable $F: \Omega \rightarrow \mathbb{R}$ we have
\begin{equation}\label{eq:ergodic-thm-for-shift}
\lim_{\tau\rightarrow \infty} \frac{1}{\tau} \int_{0}^{\tau} F(\theta_t(\omega)) \, dt  = \mathbf{E} [ F ],
\end{equation}
for $\mathbf{P}$-almost all $\omega \in \Omega$.

Under our coercive hypothesis on the drift coefficient of equation \eqref{eq:sde-on-R}, the following limit theorem is borrowed from \cite[Section 4]{GKN}. Notice that its proof is based on contraction properties for the stochastic equation.

\begin{theorem}\label{thm:attractor}
Let  Hypothesis \ref{hyp:f-coercive} hold, and consider the unique solution $Y$ to equation~\eqref{eq:sde-on-R} as given in Proposition \ref{prop:exist-unique}. Then  there exists a random variable  $\overline{Y}: \Omega \rightarrow \mathbb{R}^d $  such that
$$   \lim_{t \rightarrow \infty}  \,\, | Y_t(\omega) - \overline{Y}(\theta_t \omega) | = 0 $$
for $\mathbf{P}$-almost all $\omega \in \Omega$. Moreover, we have
$ \mathbf{E} [|\overline{Y}|^p] < \infty $
for all $p \geq 1$.
\end{theorem}

With Theorem \ref{thm:attractor} in hand, let us label a notation for further use.

\begin{notation}\label{not-bar-Yt}
In the sequel we will denote the stationary process $\overline{Y}(\theta_t \omega)$ by $\overline{Y}_{t}$.
\end{notation}

It is worth observing that the convergence of $Y$ towards $\overline{Y}$ can also be quantified in H\"older norm:
\begin{proposition}\label{prop:cvgce-Y-bar-Y-holder}
Assume Hypothesis \ref{hyp:f-coercive} holds true.
Let $\beta\in(0,H)$. There exists a random variable $Z$ admitting moments of any order and 3 constants $c_{1},c_{2},c_{3}>0$ such that for all $0\le s\le t$ we have:
\begin{equation*}
\lln Y_{t} - \by_{t}\rrn \le Z \, e^{-c_{1}  t}
\quad\text{and}\quad
\lln\der\lc Y-\by \rc_{st}\rrn \le c_{2} Z \, e^{-c_{3}  s} (t-s)^{\beta}.
\end{equation*}
\end{proposition}

\begin{proof}
The difference $Y_{t} - \by_{t}$ satisfies
$$
Y_{t} - \by_{t}=y_0 - \by_{0}+ \iot  [b(Y_s;\te)-b(\by_s;\te)] \, ds
$$
Applying the change of variable formula in the Young setting (\ref{ito_proto}) and using Hypothesis~\ref{hyp:f-coercive}(i), we get:
\begin{equation*} \begin{split}
\lln Y_{t} - \by_{t}\rrn^2 &=\lln y_0 - \by_{0} \rrn^2+ 
2\iot  \langle b(Y_s;\te)-b(\by_s;\te), Y_{s} - \by_{s}\rangle \, ds \\
&\leq \lln y_0 - \by_{0} \rrn^2-2\alpha \int_0^t \lln Y_{s} - \by_{s}\rrn^2\, ds.
\end{split}
\end{equation*}
The first claim follows by a direct application of Gronwall's lemma.

Similarly, for our additional parameter $\beta\in(0,H)$, we have
$$
\der\lc Y-\by \rc_{st}=\int_s^t  [b(Y_u;\te)-b(\by_u;\te)] \, du
\le c_{b} \lp I_{st} \rp^{1-\beta} I_{st}^{\beta},
$$
where $c_{b}$ is the uniform bound on the Lipschitz constant of $b$, and where
\begin{equation*}
I_{st} = \int_{s}^{t} \lln Y_{u} - \by_{u}\rrn \, du.
\end{equation*}
We now bound $I_{st}$ in two different ways:

\noindent\emph{\textnormal{(i)}} Since we have just proved that $| Y_{u} - \by_{u}| \le Z \, e^{-c_{1}  u}$, we get
\begin{equation}\label{a01}
I_{st} 
\le
Z \int_{s}^{t} e^{-c_{1}  u} \, du
\le
Z \int_{0}^{\infty} e^{-c_{1}  u} \, du
=
\frac{Z}{c_{1}}.
\end{equation}

\noindent\emph{\textnormal{(ii)}} 
Still using the relation $| Y_{u} - \by_{u}| \le Z \, e^{-c_{1}  u}$, it is also readily checked that:
\begin{equation*}
I_{st} 
\le
Z \int_{s}^{t} e^{-c_{1}  u} \, du
\le
Z \, e^{-c_{1}  s} (t-s).
\end{equation*}

\noindent
Plugging those two bounds into \eqref{a01}, we get
\begin{equation*}
I_{st} 
\le
\frac{Z \, e^{-c  \beta s}}{c^{1-\beta}} (t-s)^{\beta},
\end{equation*}
from which our second claim is easily deduced.
\end{proof}

We now recall, similarly to \cite{NT}, that the integrability of $\overline{Y}$ 
 implies the  ergodicity of equation (\ref{eq:sde-on-R}):
\begin{proposition}\label{prop:ergod-Y}
Assume Hypothesis \ref{hyp:f-coercive} holds true. 
Consider a finite horizon $\rho>0$,  a generic parameter $\te\in\tte$, and the set $C^{\beta}([-\rho,0])$ of H\"older continuous functions on the interval $[-\rho,0]$ for an exponent $\beta<H$. Let $F$ be a functional on $C^{\beta}([-\rho,0])$ such that $F$ is Frechet differentiable and 
\begin{equation*}
|F(g)| + \|DF(g)\| 
\le
c \lp  1+|g|_{\beta;[-\rho,0]}^N\rp, \qquad g \in C^{\beta}([-\rho,0]) ,
\end{equation*}
for some $c>0$, $N \in \mathbb{N}$. Then the following limit holds true:
\begin{equation}\label{eq:lim-int-f-Y}
  \lim_{\tau\rightarrow \infty} \frac{1}{\tau} \int_{0}^{\tau} F(Y|_{[t-\rho,t]}) \, dt  
  = \mathbf{E}[f(\overline{Y})|_{[-\rho,0]}], \qquad  \mathbf{P}\textrm{-}a.s.
\end{equation}
\end{proposition}

\begin{proof}
This is a direct consequence of Proposition \ref{prop:cvgce-Y-bar-Y-holder} and the fact that relation \eqref{eq:ergodic-thm-for-shift} holds true for the ergodic system $(w,\overline{Y})$.
\end{proof}

\section{Proof of Theorem \ref{lan}}\label{sec:proof-main-thm}

We can now gather the information we have obtained on the system \eqref{eq:sde-on-R} in order to complete the proof of Theorem \ref{lan}. This will be divided in different steps.

\medskip

\noindent
{\it Step 1.} The first step of the proof consists in applying  Girsanov's theorem. 

We first notice that $D_{+}^{H-1/2}b(Y; \theta)$ is well defined on $[0,\tau]$, which follows since  
 the function $u\mapsto b(Y_u; \theta)$ is $(H-\ep)$-H\"older, and thus $(H-1/2)$-H\"older.
  
 Fix $\theta \in \Theta$, and set $\te_{\tau}=\theta+\tau^{-1/2} u$. 
Suppose that   
Novikov's condition (see \cite[Theorem 7.1.1]{F75})
 holds. Then  our
 Girsanov-type theorem \ref{girsanov}
  implies that
\begin{equation}\label{eq:radon-nyk-1} 
\begin{split}
\log \lp\frac{d {\bp}^\tau_{\te_{\tau}}}{d {\bp}^\tau_{\theta}}\rp &=
- \frac{1}{c_{2}(H)}
\int_0^\tau \langle \sigma^{-1} ([D_{+}^{H-1/2} b(Y;\theta)]_t - [D_{+}^{H-1/2}b(Y;\te_{\tau})]_t),  dW_t \rangle
\\
& \qquad -\frac{1}{2(c_{2}(H))^{2}} \int_0^\tau  \vert \sigma^{-1} ([D_{+}^{H-1/2} b(Y;\theta)]_t-[D_{+}^{H-1/2}b(Y;\te_{\tau})]_t) \vert^2 dt.
\end{split}
\end{equation}

In the following, we show that Novikov's condition holds for $\bar{b}_{t} := b(Y_{t};\theta)- b(Y_{t}; \theta_{\tau})$, that is, that there exists $\lambda>0$ such that
\begin{equation} \label{novikov}
\sup_{ t\in [0,\tau]} \be_{\theta}\! \left[ \exp\left( \lambda \int_0^t \vert \sigma^{-1} [D_+^{H-1/2} \bar{b} ]_s \vert^2 ds \right)\right] < \infty.
\end{equation}
Towards this aim, we extend the definition of $Y$ to $\R_{-}$ by setting $Y_{t}=y_{0}$ for all $t\le 0$.
Then,
\begin{equation}\label{e.bdDy} \begin{split}
\vert  \sigma^{-1} [D_+^{H-1/2} \bar{b} ]_s\vert =c_H \bigg\vert\int_{\R_{+}} \frac{\sigma^{-1} (\bar{b}_{s} -\bar{b}_{s-r} )}{r^{H+1/2}} \, dr \bigg\vert\leq c_H(A_{1,s}+ A_{2,s}),
 \end{split}
\end{equation}
where
\begin{equation*}
\begin{split}
A_{1,s}= \bigg\vert\int_0^s \frac{\sigma^{-1} (\bar{b}_{s} -\bar{b}_{s-r} )}{r^{H+1/2}} \, dr \bigg\vert\quad \text{ and } \quad
A_{2,s}=\bigg\vert\int_s^{\infty} \frac{\sigma^{-1} (\bar{b}_{s} -\bar{b}_{0} )}{r^{H+1/2}} \, dr \bigg\vert.
\end{split}
\end{equation*}
Using the mean value theorem and the uniform boundedness of $\partial_x b$, the second term is easily bounded as: 
\begin{equation*}
\begin{split}
A_{2,s}\leq c_H \frac{\vert Y_s-y_0 \vert}{s^{H-1/2}}.
\end{split}
\end{equation*}
Equation \eqref{eq:sde-on-R}, the linear growth of $b$ and Lemma \ref{bounds} yield
$$
\vert Y_s-y_0 \vert \leq cs(1+ e^{cs} \sup_{r \in [0,s]} \vert B_t \vert)+c \vert B_s \vert.
$$
Therefore,
\begin{eqnarray}\label{eqn.A2int}
\int_{0}^{t} |A_{2,s}|^{2} ds \leq
\int_0^t  \frac{\vert Y_s-y_0 \vert^2}{s^{2H-1}} ds
\leq c e^{ct}\left( 1+ | B |_{\infty;[0,t]}^2\right). 
\end{eqnarray}

Similarly, we get that for all $0<\ep<\frac12$,
\begin{equation*}
\begin{split}
 A_{1,s}  
 \leq c  \sup_{0< r\leq s} \frac{\vert \delta Y_{s-r, s}  \vert}{r^{H-\ep}} \int_0^s \frac{1}{r^{\frac12+\ep}} dr  =  c_H s^{\frac12-\ep} \sup_{0< r\leq s} \frac{\vert \delta Y_{s-r, s}  \vert}{r^{H-\ep}}.
\end{split}
\end{equation*}
Moreover, our bound \eqref{eqn.bddY} yields:
$$
\vert \delta Y_{s-r, s}  \vert \leq cr (1+e^{cs} \sup_{u \in [0,t]} \vert B_u\vert )+ c\vert \delta B_{s-r, s}  \vert,
$$
and thus $A_{1,s}$ satisfies the following bound:
\begin{eqnarray}\label{eqn.A1int}
\int_0^t  A_{1,s}^2 ds
\leq c e^{ct} \left( 1+| B |_{\infty;[0,t]}^2+|B |_{H-\epsilon;[0,t]}^2\right). 
\end{eqnarray}
Plugging our estimates \eqref{eqn.A2int} and \eqref{eqn.A1int} into \eqref{e.bdDy}, we have thus obtained 
\begin{eqnarray*}
\int_{0}^{t} |\si^{-1} [D^{H-1/2}_{+} \bar{b} ]_{s}|^{2} ds &\leq& ce^{ct}( 1+ |B|_{\infty; [0,t]}^{2}+ |B|_{H-\ep; [0,t]}^{2} ).
\end{eqnarray*}

\noindent Finally, Fernique's theorem \cite{F} implies (\ref{novikov}).

\noindent
{\it Step 2.} We next linearize relation \eqref{eq:radon-nyk-1}. To this aim, recall that $\hb$ is defined in Hypothesis~\ref{hyp:f-coercive} as $\hb=\partial_{\te}b$. We add and substract the $d$-dimensional vector 
$$ [D_{+}^{H-1/2} \hat{b}(Y; \theta)]_t (\theta_{\tau}-\te)
=
\frac{1}{\sqrt{\tau}} [D_{+}^{H-1/2} \hat{b}(Y; \theta)]_t u,
$$ 
where we abbreviate $\langle [D_{+}^{H-1/2} \hat{b}(Y; \theta)]_t, \, u\rangle_{\R^{q}}$ into $[D_{+}^{H-1/2} \hat{b}(Y; \theta)]_t u$ for notational sake.
This easily yields:
\begin{equation}\label{eq:dcp-log-lik}
\log \lp\frac{d {\bp}^\tau_{\te_{\tau}}}{d {\bp}^\tau_{\theta}}\rp =I_1+I_2-\frac12 I_3-I_4,
\end{equation}
where 
\begin{equation*} \begin{split}
&I_1=  \frac{1}{\tau^{1/2} \, c_{2}(H)}\int_0^\tau \langle \sigma^{-1} [D_{+}^{H-1/2} \hat{b}(Y; \theta)]_t u,  dW_t \rangle
-\frac{1}{2(c_{2}(H))^{2}\tau} \int_0^\tau  \vert \sigma^{-1}  [D_{+}^{H-1/2} \hat{b}(Y; \theta)]_t u\vert^2 dt
\\
&I_2= \frac{1}{c_{2}(H)}
\int_0^\tau \langle  \sigma^{-1} ([D_{+}^{H-1/2}b(Y;\te_{\tau})]_t-[D_{+}^{H-1/2} b(Y;\theta)]_t- [D_{+}^{H-1/2} \hat{b}(Y; \theta)]_t (\te_{\tau}-\theta)),  dW_t \rangle
\\
& I_3=  \frac{1}{(c_{2}(H))^{2}}
\int_0^\tau  \vert  \sigma^{-1} ([D_{+}^{H-1/2}b(Y;\te_{\tau})]_t-[D_{+}^{H-1/2} b(Y;\theta)]_t-[D_{+}^{H-1/2} \hat{b}(Y; \theta)]_t (\te_{\tau}-\theta))\vert^2 dt 
\end{split}
\end{equation*}
and 
\begin{multline*}
I_4=   \frac{1}{(c_{2}(H))^{2}}
\int_0^\tau  \langle \sigma^{-1} ([D_{+}^{H-1/2}b(Y;\te_{\tau})]_t-[D_{+}^{H-1/2} b(Y;\theta)]_t\\
 - [D_{+}^{H-1/2} \hat{b}(Y; \theta) ]_t (\theta_{\tau}-\te)),
 \sigma^{-1}[D_{+}^{H-1/2}\hat{b}(Y; \theta)]_t (\te_{\tau}-\theta)\rangle  dt.
\end{multline*}

\noindent
{\it Step 3.} In this step we set the ground for the identification of the main contribution to our log-likelihood. That is, we wish to show that as $\tau \rightarrow \infty$ we have: 
\begin{equation} \label{step3}
\frac{1}{\tau (c_{2}(H))^{2}} \int_0^\tau  \vert \sigma^{-1}  [D_{+}^{H-1/2} \hat{b}(Y; \theta)]_t u \vert^2 dt \overset{{\bp}_{\theta}}{\longrightarrow} u^{\rm T}\Sigma(\theta) u,
\end{equation}
where $\Sigma(\theta)$ is given by (\ref{gamma_def}).
Observe that (\ref{step3}) together with the multivariate central limit theorem for Brownian martingales (cf. \cite[Theorem 1.21]{KY}), imply that $I_1$ is the term that  contributes to the limit in (\ref{limit}). In Steps 4 and 5 below we will show that $I_2, I_3$ and $I_4$ are negligible terms with respect to this main contribution.

In order to prove (\ref{step3}), we first write
 \begin{eqnarray*}
\frac{1}{\tau (c_{2}(H))^{2}} \int_0^\tau  \vert \sigma^{-1}  [D_{+}^{H-1/2} \hat{b}(Y; \theta)]_t u \vert^2 dt 
&=& 
u^{T}  I(\tau)  u, 
\end{eqnarray*}
where we have set
\begin{eqnarray}\label{e2i}
I(\tau) &=& \frac{1}{\tau (c_{2}(H))^{2}} \int_0^\tau  
 [D_{+}^{H-1/2} \hat{b}(Y; \theta)^{T}]_t 
  (\sigma^{-1})^{T} \sigma^{-1}  [D_{+}^{H-1/2} \hat{b}(Y; \theta)]_t   dt .
\end{eqnarray}
To prove the convergence \eqref{step3} we are then reduced to show the limit 
\begin{equation}\label{eq:step3-reduced}
\lim_{\tau\to \infty} I(\tau) = \Sigma(\theta).
\end{equation} 

Next recall that for $\al\in(0,1)$ and $\vp\in\cac_{c}^{\infty}(\R)$ the fractional derivative operator $D_{+}^{\al}$ is defined by \eqref{eq:def-frac-deriv-intg}. In particular we can write:
\begin{eqnarray}\label{e8}
 [D_{+}^{H-1/2} \hat{b}(Y; \theta)]_t &=& \frac{H-1/2}{\Gamma( 3/2-H)}  \int_{\mathbb{R}_{+}} \frac{\hat{b}(Y_{t}; \theta)- \hat{b}(Y_{t-r}; \theta)}{r^{H+1/2}} dr .
\end{eqnarray}
Now substituting  \eqref{e8}  into \eqref{e2i} we obtain:
\begin{eqnarray}\label{a}
I(\tau) &=&
\frac{c_{4}(H)}{\tau} \int_0^\tau \int_{\mathbb{R}_{+}^{2}} 
\frac{\ce_{t}(r,\tilde{r})}{\tilde{r}^{H+1/2} r^{H+1/2}}     \,  drd\tilde{r}  dt  ,
\end{eqnarray}
where  we have denoted the constant $
c_{4}(H) = (\frac{(H-1/2)}{\Gamma( 3/2-H) c_{2}(H)})^{2}$, and where $\ce$ stands for the function:
\begin{eqnarray}\label{e5}
\ce_{t}(r,\tilde{r})  &=& 
 \lp  \hat{b}(Y_{t}; \theta)- \hat{b}(Y_{t-\tilde{r}}; \theta)   \rp^{T} 
  (\sigma^{-1})^{T} \sigma^{-1}  \lp  \hat{b}(Y_{t}; \theta)- \hat{b}(Y_{t-r}; \theta)  \rp.
\end{eqnarray}
We now specify the constant $c_{4}(H)$ appearing in \eqref{a}. Indeed, owing to the definition of $c_{2}(H)$ given in Proposition \ref{prop:K-H} and resorting to the elementary identity:
\begin{eqnarray}\label{e7}
\Gamma(z)\Gamma(1-z)\sin(\pi z)=\pi,\quad z>-1,
\end{eqnarray}
we get the following alternative expression for $c_{4}(H)$:
\begin{equation}\label{eq:def-c4(H)}
c_{4}(H) =
\lp  \frac{ \sin(\pi (H-\frac12))}{\pi \, c_{1}(H)}  \rp^{2}.
\end{equation}
In addition, taking into account the expression for the constant $c_{1}(H)$ given in relation \eqref{eq:def-c1(H)}, it is readily checked that $c_{4}(H)=C_{H}$, where $C_{H}$ is the constant displayed in our main Theorem \ref{lan}.  

We now further analyze $I(\tau)$ as given by \eqref{a} by exchanging the order of the integrals. Specifically, we introduce the following notation:
 \begin{eqnarray}\label{e.defJtau}
\mu (dr,d\tilde{r}) = \frac{1}{\tilde{r}^{H+1/2}}   \frac{1}{r^{H+1/2}} drd\tilde{r}  \quad\text{and}\quad J_{\tau}(r,\tilde{r}) &=&    \frac{1}{\tau} \int_0^\tau \ce_{t}(r,\tilde{r})    dt .
\end{eqnarray}
Then a standard application of Fubini's theorem to relation \eqref{a}, plus the fact that $c_{4}(H)=C_{H}$, yield: 
 \begin{eqnarray}\label{e9}
I(\tau)
&=& C_{H}  \int_{\mathbb{R}_{+}^{2}}   J_{\tau}(r,\tilde{r}) \mu(dr,d\tilde{r} ) .
\end{eqnarray}
Moreover,
sending $\tau$ to $\infty$ and invoking Proposition \ref{prop:ergod-Y} for the term $\ce_{t}(r,\tilde{r})$ defined by~\eqref{e5}, we get the following almost sure limit:
\begin{eqnarray}\label{e10i}
\lim_{\tau\to\infty} J_{\tau}(r,\tilde{r} ) = J (r,\tilde{r} )  ,
\end{eqnarray}
where  we denote:
\begin{eqnarray*}
 J (r,\tilde{r} ) = \be_{\theta}  \left[ \lp  \hat{b}(\by_{0}; \theta)- \hat{b}(\by_{-\tilde{r}}; \theta)   \rp^{T} 
  (\sigma^{-1})^{T} \sigma^{-1}  \lp  \hat{b}(\by_{0}; \theta)- \hat{b}(\by_{-r}; \theta)  \rp \right].
\end{eqnarray*}
Provided one can pass the limit through the integral in \eqref{e9}, we thus get the almost sure limit:
\begin{equation}\label{e.def-I}
\lim_{\tau\to\infty} I(\tau)
= I :=
C_{H}\int_{\mathbb{R}_{+}^{2}}   J(r,\tilde{r}) \mu(dr,d\tilde{r} ) ,
\end{equation}
and the right hand side above is the announced expression \eqref{gamma_def} for $\Sigma(\te)$. Notice that, thanks to relation \eqref{eq:pathwise-incr-bnd-Y} for small $r$, $\tilde{r}$ and Cauchy-Schwarz's inequality for large $r$ or $\tilde{r}$, it can be easily checked that $I$ is a convergent integral. 

\noindent
{\it Step 4.}
Summarizing our considerations,
in order to prove \eqref{step3} we are now reduced to take limits in $\tau$ in relation \eqref{e9}. 
To this aim, take an arbitrary constant $\ep>0$. In the following, we show that there exists  $\tau(\omega)<\infty$,   $ {\bp}_{\theta}$ a.s.  such that $|I(\tau)-I| <\ep$ for all $\tau>\tau(\omega)$, where $I$ is defined by \eqref{e.def-I}. 

Since $I$ is a convergent integral, we can take
  a constant $A=A(\ep)>0$ such that 
\begin{eqnarray}\label{e.defA}
\left| I- C_{H}\int_{[0,A]^{2} }   J(r,\tilde{r}) \mu(dr,d\tilde{r} )    \right|<\ep \quad \text{and} \quad \int_{A}^{\infty } r^{-H-1/2} dr<\ep.
\end{eqnarray} 
For $V, V' \subset \R_{+}$ we  set
\begin{eqnarray}\label{e.IVV}
I(\tau, V \times V') = C_{H}  \int_{V\times V'}   J_{\tau}(r,\tilde{r}) \mu(dr,d\tilde{r} ).
\end{eqnarray}
In the following, we   show that one can take limits in \eqref{e9} if the integral domain $\R_{+}^{2}$ is replaced by $[0,A]^{2}$. Namely, we will prove that  
\begin{eqnarray}\label{eqn.Itruncated}
\lim_{\tau\to \infty}I(\tau, [0,A]^{2}) &=& C_{H}\int_{[0,A]^{2}}   J(r,\tilde{r}) \mu(dr,d\tilde{r} ). 
\end{eqnarray}

We first derive some estimates for   $ \ce_{t}(r, \tilde{r})  $ and $J_{\tau}(r, \tilde{r})$, where we recall   that $ \ce_{t}(r, \tilde{r})  $ is defined in~\eqref{e5}.  Take $0<\delta<1/2$.   Proposition \ref{prop:pathwise-incr-bnd-Y} on the $(H-\delta)$-H\"older continuity of the solution process $Y$   allows to write the   relation    
\begin{eqnarray}\label{e.Ebound}
|\ce_{t}(r, \tilde{r}) | &\leq& K |Y|_{H-\delta, [t-A, t]}^{2} r^{H-\delta}\tilde{r}^{H-\delta},
\end{eqnarray}
which holds true  for all $0\leq r, \tilde{r} \leq A$. 
Applying  relation \eqref{e.Ebound} to $J_{\tau}(r, \tilde{r})$ in \eqref{e.defJtau}, we   then get the following upper bound for $J_{\tau}(r, \tilde{r}) $:
\begin{eqnarray}\label{eqn.Jrsmall}
|J_{\tau}(r, \tilde{r}) | &\leq& \frac{1}{\tau} \int_{0}^{\tau}K |Y|_{H-\delta, [t-A, t]}^{2} r^{H-\delta}\tilde{r}^{H-\delta}  d t.
\end{eqnarray}
Notice that Proposition \ref{prop:ergod-Y}   implies the almost sure convergence
\begin{eqnarray*}
\lim_{\tau\to\infty}
\frac{1}{\tau} \int_{0}^{\tau}K |Y|_{H-\delta, [t-A, t]}^{2}    d t = K    \be  \left[ |\by|_{H-\delta, [ -A, 0]}^{2} \right].
\end{eqnarray*}
  In particular, we can find $\tau_{1}(\omega)>0$ such that  for $\tau>\tau_{1}(\omega)$
\begin{eqnarray}\label{e.gdominate}
\frac{1}{\tau} \int_{0}^{\tau}K  |Y|_{H-\delta, [t-A, t]}^{2}  r^{H}\tilde{r}^{H}  d t   \leq g(r, \tilde{r}):= \lp 1+ K   \be  \left[  |\by|_{H-\delta, [ -A, 0]}^{2} \right] \rp r^{H-\delta}\tilde{r}^{H-\delta} . 
\end{eqnarray}
So the estimate \eqref{eqn.Jrsmall}  implies that for $\tau>\tau_{1}(\omega)$
\begin{eqnarray}\label{e.gboundJ}
|J_{\tau}(r, \tilde{r}) | \leq g(r, \tilde{r}).
\end{eqnarray}

On the other hand, 
it is clear     that 
$\int_{[0,A]^{2}} g(r, \tilde{r}) \mu(dr, d\tilde{r}) <\infty$.
Thus taking into account relation \eqref{e10i} and \eqref{e.gboundJ}, one can apply the dominated convergence theorem in \eqref{e.IVV} with $V=V'=[0,A] $. This yields   \eqref{eqn.Itruncated}. 
In the sequel, we will thus consider   $ \tau_{2} (\omega) <\infty$ such that   for $\tau>\tau_{2}(\omega)$ we have 
 \begin{eqnarray}\label{e.IA.error}
 \left|
I(\tau, [0,A]^{2}) - C_{H}\int_{[0,A]^{2}}   J(r,\tilde{r}) \mu(dr,d\tilde{r} )
\right| &<& \ep. 
\end{eqnarray}

In the following, we show that $|I(\tau) - I(\tau, [0,A]^{2})|<\ep$ for  $\tau$ sufficiently large.   We first show that the quantities $ \frac{1}{\tau} \int_{0}^{\tau} |\hat{b} (Y_{t}; \theta ) |^{2} dt$ and $ \frac{1}{\tau} \int_{0}^{\tau} |\hat{b} (Y_{t-r}; \theta ) |^{2} dt$ can be bounded by some constants. 
   In fact, by Proposition  \ref{prop:ergod-Y} again we   have the almost sure convergence
\begin{eqnarray*}
\lim_{\tau\to\infty} \frac{1}{\tau} \int_{0}^{\tau} |\hat{b} (Y_{t}; \theta ) |^{2} dt = \be  |\hat{b} (\by_{0}; \theta ) |^{2} .
\end{eqnarray*}
In particular, there exists    $\tau_{3} (\omega) $ satisfying $ \tau_{2}(\omega)\leq \tau_{3}(\omega)<\infty$, such that the following upper bound estimate holds true for all $\tau>\tau_{3}(\omega)$:
\begin{eqnarray}\label{e.boundbt}
\frac{1}{\tau} \int_{0}^{\tau} |\hat{b} (Y_{t}; \theta ) |^{2} dt  &\leq& K_{1}:= 1+ \be  |\hat{b} (\by_{0}; \theta ) |^{2}. 
\end{eqnarray} 
On the other hand, 
due to the fact that $Y_{u}=Y_{0}$ for $u\leq 0$, the following holds true for 
  $r\leq \tau$:
\begin{eqnarray}\label{e.bound.bt-r}
\frac{1}{\tau}\int_{0}^{\tau} |\hat{b}(Y_{t- {r}}; \theta)   |^{2}dt &=&\frac{1}{\tau} \int_{r}^{\tau} |\hat{b}(Y_{t- {r}}; \theta)   |^{2}dt +\frac{1}{\tau}\int_{0}^{r} |\hat{b}(Y_{0}; \theta)   |^{2}dt 
\nonumber
\\
&
\leq & K_{1}+ |\hat{b}(Y_{0}; \theta)   |^{2}:=K_{2},
\end{eqnarray}
while   \eqref{e.bound.bt-r} is also trivially true for $r> \tau$. 
 In conclusion, the estimate \eqref{e.bound.bt-r} is valid for all $\tau\geq \tau_{3}(\omega)$ and all $r\geq 0$. 
 
  We will now bound the quantity $|I(\tau) - I(\tau, [0,A]^{2})|$ by sums of elements of the form   $I(\tau, V\times V')$. 
  Recall  that $I(\tau, V\times V')$ is defined in \eqref{e.IVV}. 
  Those elements will then be estimated thanks to \eqref{e.boundbt} and \eqref{e.bound.bt-r}. To this aim, we first recall that $ J_{\tau}(r, \tilde{r}) $ is defined by~\eqref{e.defJtau}. Therefore, an elementary application of H\"older's inequality yields  
\begin{eqnarray}\label{e.Jholder}
|J_{\tau}(r, \tilde{r})|  \leq   K \lp \frac{1}{\tau} \int_{0}^{\tau} |  \hat{b}(Y_{t}; \theta)- \hat{b}(Y_{t- {r}}; \theta)  |^{2}dt \rp^{1/2} \lp \frac{1}{\tau} \int_{0}^{\tau} |  \hat{b}(Y_{t}; \theta)- \hat{b}(Y_{t-\tilde{r}}; \theta)   |^{2}dt \rp^{1/2}   .
\end{eqnarray}
Then applying   \eqref{e.Jholder} to the definition \eqref{e.IVV} of $I(\tau, V\times V')$ we obtain
\begin{eqnarray}\label{e.IVVbound}
|I(\tau, V\times V')| &\leq& K F_{V}F_{V'}, 
\end{eqnarray}
where the quantity $F_{V}$ is defined by 
\begin{eqnarray}\label{e.Fv.def}
F_{V} &=& \int_{V}  \lp \frac{1}{\tau} \int_{0}^{\tau} |  \hat{b}(Y_{t}; \theta)- \hat{b}(Y_{t- {r}}; \theta)  |^{2}dt \rp^{1/2} r^{-H-1/2} dr .
\end{eqnarray}

As a last preliminary step, let us bound trivially $\hat{b}(Y_{t}; \theta)- \hat{b}(Y_{t- {r}}; \theta)$ as follows:
\begin{eqnarray*}
| \hat{b}(Y_{t}; \theta)- \hat{b}(Y_{t- {r}}; \theta) |^{2}   \leq   2 | \hat{b}(Y_{t}; \theta)  |^{2} +2 | \hat{b}(Y_{t-r}; \theta)  |^{2} .
\end{eqnarray*}
 So invoking relations \eqref{e.boundbt} and \eqref{e.bound.bt-r}, 
we obtain the estimate
\begin{eqnarray}\label{e.Fv}
F_{V  } &\leq& 2 \int_{V  }  \lp \frac{1}{\tau} \int_{0}^{\tau} \lp |  \hat{b}(Y_{t}; \theta)   |^{2} + |  \hat{b}(Y_{t-r}; \theta)   |^{2}\rp dt \rp^{1/2} r^{-H-1/2} dr 
\nonumber
\\
&\leq& 2 (K_{1} + K_{2} )^{1/2} \int_{V  }  r^{-H-1/2} dr  
 .
\end{eqnarray}
We now analyze $F_{V}$ for three different choices of set $V$. 
 
\noindent
(i) Take $V_{1}=(A, \infty)$. Then applying \eqref{e.Fv} and taking into account   the second relation in~\eqref{e.defA}  we obtain
\begin{eqnarray}\label{e.FV1}
F_{V_{1}}  
 &\leq& 2(K_{1} + K_{2} )^{1/2} \ep := K_{3}\ep .
\end{eqnarray}

\noindent
(ii) Similarly,   consider $V_{2} = (1,A)$, then the estimate \eqref{e.Fv} implies 
\begin{eqnarray}\label{e.FV2}
F_{V_{2}} &\leq& K( K_{1}  + K_{2} )^{1/2} :=K_{4}.
\end{eqnarray}

\noindent
(iii)
Take now $V_{3} = (0,1)$. 
Similarly to \eqref{e.Ebound}, for $r\in V_{3}$ we have: 
\begin{eqnarray*}
 |  \hat{b}(Y_{t}; \theta)- \hat{b}(Y_{t-r}; \theta)  |^{2} &\leq& K |Y|_{H-\delta, [t-1,t]}^{2}\,r^{2(H-\delta)}.
\end{eqnarray*}
Hence
invoking expression \eqref{e.Fv.def} of $F_{V}$ and 
along the same lines as for   \eqref{e.gdominate} we obtain
\begin{eqnarray}\label{e.FV3}
F_{V_{3}}  &\leq& K \lp \frac{1}{\tau} \int_{0}^{\tau}  |Y|_{t-1,t}^{2} dt \rp^{1/2} \int_{V_{3}}  r^{-1/2-\delta} dr
 \leq K (1+ \be  | \by |_{-1,0}^{2}  )^{1/2}:= K_{5},
\end{eqnarray}
for all $\tau>\tau_{3}(\omega)$. 

 With these preparations,  we now consider the difference  $ I(\tau) - I(\tau, [0,A]^{2}) $ for $\tau>\tau_{3}(\omega)$. Owing to  relation \eqref{e.IVVbound} and recalling that $V_{1} =(A, \infty)$ , $V_{2}=(1,A)$ and $V_{3}=(0,1)$ we   have
 \begin{eqnarray*}
| I(\tau) - I(\tau, [0,A]^{2}) | &=&| I(\tau, V_{1}\times (0,A))+I(\tau,  (0,A)\times V_{1}) +I(\tau,  V_{1} \times V_{1}) |
\\
&\leq& 2F_{V_{1}} (F_{V_{2}}+F_{V_{3}})+   F_{V_{1}}^{2}.
\end{eqnarray*}
We then apply \eqref{e.FV1}, \eqref{e.FV2}, \eqref{e.FV3}, which yields  
\begin{eqnarray}\label{e.Itail}
| I(\tau) - I(\tau, [0,A]^{2}) | \leq  2K_{3}(K_{4}+K_{5}) \ep+   K_{3}^{2}\ep^{2} \leq  K\ep.
\end{eqnarray}

We can now conclude the proof of \eqref{step3}. Indeed, combining \eqref{e.defA}, \eqref{e.IA.error} and \eqref{e.Itail}, we easily get the estimate
\begin{eqnarray*}
|I(\tau) - I|< K\ep \quad \text{for} \quad \tau>\tau_{3}(\omega), 
\end{eqnarray*}
from which \eqref{step3} is trivially deduced. Notice that we have obtained in fact a stronger statement than \eqref{step3}, since our limit holds in the almost sure sense.

\noindent
{\it Step 5.}  We next show that the term $I_3$ in \eqref{eq:dcp-log-lik} converges to zero in $\bp_{\theta}$-probability as $\tau \rightarrow \infty$. 
For this, set $g_{t}(\te)=\si^{-1}[D_{+}^{H-1/2}b(Y;\te)]_t$. One can recast $I_3$ into:
\begin{equation*}
I_{3}
=
\int_{0}^{\tau} \lln  g_{t}(\te_{\tau}) - g_{t}(\te) - \partial_{\te}g_{t}(\te) (\te_{\tau}-\te) \rrn^{2} \, dt,
\end{equation*}
and we also recall that $\te_{\tau}-\te=\tau^{-1/2} u$. In addition, a simple application of Taylor's expansion for multivariate function  yields the existence of a $\la\in(0,1)$ such that:
\begin{equation*}
g_{t}(\te_{\tau}) - g_{t}(\te)
=
\partial_{\te} g_{t}(\xi_{\la}) (\theta_{\tau}-\theta),
\quad\text{where}\quad
\xi_{\la} = \te + \la (\te_{\tau}-\te).
\end{equation*}
Notice that under our standing Hypothesis \ref{hyp:f-coercive} we have $\partial_{\te}g_{t}(\xi_{\lambda})=\si^{-1}[D_{+}^{H-1/2}\hb(Y;\xi_{\lambda})]_t$. We thus get:
\begin{equation*}
I_3 =\frac{1}{\tau}
\int_0^\tau \vert \sigma^{-1} M_t(Y) u \vert^2 \, dt, \quad
\text{where} \quad
M_t(Y)= \int_{\R_{+}}  
\frac{\delta[\hb(Y;\xi_{\lambda})-\hb(Y;\theta)]_{t-r,t}}{r^{1/2+H}} \, dr.
\end{equation*}
Next we decompose $M_{t}(Y)$ into $M_{1,t}(Y)+M_{2,t}(Y)$, where
\begin{equation*}
M_{1,t}(Y)= \int_0^t
\frac{ \delta [\hb(Y;\xi_{\lambda})-\hb(Y;\theta)]_{t-r,t}}{r^{1/2+H}} \, dr,
\qquad
M_{2,t}(Y)= \int_t^{\infty}  
\frac{\delta [\hb(Y;\xi_{\lambda})-\hb(Y;\theta)]_{t-r,t}}{r^{1/2+H}} \, dr.
\end{equation*}
Recall that we have extended the definition of $Y$ to $\R_{-}$ by setting $Y_{t}=y_{0}$ for all $t\le 0$. 
Therefore a simplified expression for $M_{2,t}(Y)$ is as follows:
\begin{equation*}
M_{2,t}(Y) 
=
\lc \hb(Y_{t};\xi_{\lambda})-\hb(Y_{t};\theta) \rc  \int_t^{\infty} \frac{dr}{r^{ H+1/2}} ,
\end{equation*}
and using Hypothesis \ref{hyp:f-coercive}
we obtain that:
\begin{equation*} \begin{split}
\tau^{-1} \int_0^\tau \vert \sigma^{-1} M_{2,t}(Y) u\vert^2 \, dt \lesssim \tau^{-2} \int_0^\tau \frac{(1+\vert Y_t \vert^2)}{t^{2H-1}} dt.
\end{split}
\end{equation*}
Thus, by Proposition  \ref{prop:bnd-moments-Y},  the $L^1(\Omega)$-norm of this term is bounded by $c_H \tau^{-2H}$. Hence, this term converges in $\bp_{\theta}$-probability to zero as $\tau \rightarrow \infty$.

On the other hand, fix $\alpha \in (\frac{1}{1+2H}, \frac{1}{2H})$, and write
\begin{equation*} \begin{split}
M_{1,t}(Y) \leq \int_0^t
\frac{\left|\delta [\hb(Y;\xi_{\lambda})-\hb(Y;\theta)]_{t-r,t} \right|^{\alpha} \cdot \left|\delta [\hb(Y;\xi_{\lambda})-\hb(Y;\theta)]_{t-r,t}\right|^{1-\alpha}}{r^{1/2+H}} \, dr.
\end{split}
\end{equation*}
Then, appealing to Hypothesis \ref{hyp:f-coercive}, we get that
\begin{equation*} \begin{split}
\vert M_{1,t}(Y) \vert \lesssim \frac{1}{\tau^{\alpha/2}}  \int_0^t
\frac{\left(1+\vert Y_t \vert^{\alpha} +\vert Y_{t-r} \vert^{\alpha}\right) \vert Y_t-Y_{t-r} \vert^{1-\alpha}}{r^{H(1-\alpha)}r^{1/2+\alpha H}} \, dr 
\end{split}
\end{equation*}
Therefore,
\begin{equation*} \begin{split}
&\vert M_{1,t}(Y) \vert^2 \, \\
&\lesssim  \frac{1}{\tau^{\alpha}} \int_0^t \int_0^t
\frac{\left(1+\vert Y_t \vert^{\alpha} +\vert Y_{t-r_1} \vert^{\alpha}\right) \vert Y_t-Y_{t-r_1} \vert^{1-\alpha}}{r_1^{H(1-\alpha)}r_1^{1/2+\alpha H}} 
\frac{\left(1+\vert Y_t \vert^{\alpha} +\vert Y_{t-r_2} \vert^{\alpha}\right) \vert Y_t-Y_{t-r_2} \vert^{1-\alpha}}{r_2^{H(1-\alpha)}r_2^{1/2+\alpha H}}\, dr_1
dr_2.
\end{split}
\end{equation*}
Thus, again by Proposition \ref{prop:bnd-moments-Y}, the $L^1(\Omega)$-norm of the term
$
\frac{1}{\tau} \int_0^\tau\vert\sigma^{-1} M_{1,t}(Y) u\vert^2 dt
$
is bounded by
$$
\frac{c_H}{\tau^{1+\alpha}}\int_0^\tau \int_0^t \int_0^t
\frac{1}{r_1^{1/2+\alpha H}} 
\frac{1}{r_2^{1/2+\alpha H}}\, dr_1
dr_2 dt=\frac{c_H}{\tau^{1+\alpha}}\int_0^\tau t^{1-2H\alpha} dt=\frac{c_H}{\tau^{\alpha (2H+1)-1}},
$$
which converges to zero as $\tau \rightarrow \infty$ since $\alpha>\frac{1}{2H+1}$.

\vskip 12pt

\noindent
{\it Step 6.} 
We finally show that $I_2$ and $I_4$ are also negligible terms. Since $I_3$ is the quadratic variation of the martingale $I_2$, this implies that $I_2$ converges to zero in $\bp_{\theta}$-probability
as $\tau \rightarrow \infty$. On the other hand, applying Cauchy-Schwarz inequality to $I_4$,
we get that
\begin{equation*}
\begin{split}
\vert I_4 \vert \leq  \left( \frac{1}{\tau}\int_0^\tau \vert \sigma^{-1}  [D_{+}^{H-1/2} \hat{b}(Y;\theta)]_tu  \vert^2 dt \right)^{1/2} \times I_3^{1/2},
\end{split}
\end{equation*}
Thus, by the results in Steps 3 and 4, we obtain that $I_4$ converges to zero in $\bp_{\theta}$-probability as $\tau \rightarrow \infty$, which completes the proof of Theorem \ref{lan}.

\section{Ornstein-Uhlenbeck case}

Though our result encompasses a wide range of coefficients $b$, it is worth illustrating it on a simple linear case, that is for  the real-valued fractional Ornstein-Uhlenbeck process corresponding to $b(x;\te)=-\te\, x$ and $\si\equiv \1$. In addition, we will assume that our parameter $\te$ is an element of the set $\Theta$, which is a compact interval in $(0, \infty)$. Specifically, the equation followed by $Y$ is the following:
\begin{equation} \label{ou}
Y_t=-\theta \int_0^t Y_s \, ds +B_t,  \qquad t\in\ott, \quad \theta>0,
\end{equation}
and the stationary solution is given by:
$$
\by_t= \int_{-\infty}^t e^{-\theta (t-s)} \, dB_s.
$$

\subsection{Computation of the LAN variance}

One of the advantages of the Ornstein-Uhlenbeck case is that it allows explicit computations of the LAN variance $\Sigma(\theta)$. We summarize this possibility in the following proposition:

\begin{proposition} \label{ousigma}
Let $Y$ be the fractional Ornstein-Uhlenbeck process solution to \textnormal{(\ref{ou})}. Then the LAN property (\ref{limit}) is satisfied with
$$
\Sigma(\theta)=\frac{1}{2\theta}.
$$
\end{proposition}

\begin{remark}
As mentioned in the introduction, 
we observe that in this case $\Sigma(\theta)$ does not depend on $H$, and is equal to the asymptotic variance of the LAN property of the Ornstein-Ulhenbeck process driven by a standard Brownian motion.
\end{remark}

\begin{proof}[Proof of Proposition \ref{ousigma}]
By (\ref{gamma_def}), we have the following expression for $\Sigma(\theta)$
\begin{equation*} \begin{split}
\Sigma(\theta)=C_H \int_{\R_{+}^{2}} \frac{\be_{\theta}[\by^2_{0}]
-\be_{\theta}[\by_{0}\by_{r_{1}}]-\be_{\theta}[\by_{0}\by_{r_{2}}]
+\be_{\theta}[\by_{0} \by_{r_2-r_{1}} ]}{r_{1}^{1/2+H} r_{2}^{1/2+H}}  \, dr_{1} dr_{2},
\end{split}
\end{equation*}
where $C_H$ is given by (\ref{ch}).

We start by computing the covariance of $\by$ between times $0$ and $t$. Namely, thanks to  \cite[(3.4)]{PT}, we obtain
\begin{equation*} 
\begin{split}
\be_{\theta}[\by_{0}\by_{t}]&=\frac{1}{  c_5(H)}\int_{\R} \vert \xi \vert^{1-2H} \mathcal{F}({\bf 1}_{(-\infty,t)} 
e^{-\theta(t-\cdot)}) (\xi) \overline{\mathcal{F}({\bf 1}_{(-\infty,0)} 
e^{\theta \cdot}) (\xi)}\, d \xi \\
&=\frac{1}{  c_5(H)}\int_{\R} \vert \xi \vert^{1-2H} \frac{e^{it\xi}}{(i\xi+\theta)} \frac{1}{(-i\xi+\theta)} \, d\xi \\
&=\frac{1}{  c_5(H)} \int_{\R} \frac{\vert \xi \vert^{1-2H}}{\theta^{2}+\xi^2} e^{it \xi} \, d \xi,
\end{split}
\end{equation*}
where $ \cf (f)(\xi) = \int_{\R} f(t) e^{it\xi} dt$  and
\begin{eqnarray}\label{eqn.c5}
c_5(H)=\frac{2 \pi}{\Gamma(2H+1) \sin(\pi H)}. 
\end{eqnarray}
Therefore, using Fubini's theorem, we obtain
\begin{align} 
\Sigma(\theta)&=\frac{C_H}{  c_5(H)}\int_{\R_+^{2}} \int_{\R} \frac{\vert \xi \vert^{1-2H}}{\theta^{2}+\xi^2}
 \frac{\left(1
-e^{i r_1 \xi}-e^{ir_2 \xi}
+e^{i(r_2-r_{1}) \xi}\right)}{r_{1}^{1/2+H} r_{2}^{1/2+H}} \, d\xi \, dr_{1} dr_{2} 
\nonumber
\\
&=\frac{C_H}{  c_5(H)}\int_{\R} \frac{\vert \xi \vert^{1-2H}}{\theta^{2} +\xi^2}
\int_{\R_+^{2}}  \frac{\left(1
-e^{i r_1 \xi}\right)\left(1-e^{-ir_2 \xi}\right)}{r_{1}^{1/2+H} r_{2}^{1/2+H}} \, dr_{1} dr_{2}\, d\xi
\nonumber
 \\
&=\frac{C_H}{  c_5(H)}\int_{\R} \frac{\vert \xi \vert^{1-2H}}{\theta^{2} +\xi^2}
\bigg\vert \int_0^{\infty}  \frac{1
-e^{i r \xi}}{r^{1/2+H} } \, dr \bigg\vert^2\, d\xi 
\nonumber
\\
&=\frac{2C_H}{  c_5(H)}\int_{0}^{\infty}  \frac{\xi^{1-2H}}{\theta^{2} +\xi^2}
\bigg\vert \int_0^{\infty}  \frac{1
-e^{i r \xi}}{r^{1/2+H} } \, dr \bigg\vert^2\, d\xi .
\label{eqn.sigma.var}
\end{align}
Thus, setting $y=r\xi$ in the integral above, we get
\begin{eqnarray*}
\Sigma (\theta) 
&=
&
\frac{2C_H}{  c_5(H)}\int_{0}^{\infty}  \frac{1}{\theta^{2} +\xi^2}\, d\xi \, 
\bigg\vert \int_0^{\infty}  \frac{1
-e^{i y}}{y^{1/2+H} } \, dy \bigg\vert^2 \\
&=& \frac{\pi C_H}{\theta c_5(H)}
\bigg\vert \int_0^{\infty}  \frac{1
-e^{i y}}{y^{1/2+H} } \, dy \bigg\vert^2.
\end{eqnarray*}

We next write
\begin{equation*}
\begin{split}
\bigg\vert \int_0^{\infty}  \frac{1
-e^{i y}}{y^{1/2+H} } \, dy \bigg\vert^2 =\left(\int_0^{\infty}  \frac{1
-\cos(y)}{y^{1/2+H} } \, dy\right)^2 +
\left(\int_0^{\infty}  \frac{\sin(y)}{y^{1/2+H} } \, dy \right)^2.
\end{split}
\end{equation*}
For the first integral on the right hand side we use the integration by parts $u=1-\cos(y)$, $dv=y^{-\frac12-H}$, $du=\sin(y)$, $v=\frac{y^{\frac12-H}}{(\frac12-H)}$.
For the second integral, we use $u=\sin(y)$, $du=\cos(y)$, to get that
\begin{equation*}
\begin{split}
&\bigg\vert \int_0^{\infty}  \frac{1
-e^{i y}}{y^{1/2+H} } \, dy \bigg\vert^2 =\left(\frac{1}{(H-\frac12)}\int_0^{\infty}  \frac{\sin(y)}{y^{H-1/2} } \, dy\right)^2 +
\left(\frac{1}{(H-\frac12)}\int_0^{\infty}  \frac{\cos(y)}{y^{H-1/2} } \, dy \right)^2 \\
&=\left(\frac{\pi}{2(H-\frac12)\Gamma(H-\frac12) \sin\left( \frac{\pi (H-\frac12)}{2}\right)}\right)^2 +
\left(\frac{\pi}{2(H-\frac12)\Gamma(H-\frac12) \cos\left( \frac{\pi (H-\frac12)}{2}\right)}\right)^2 \\
&=\frac{\pi^2}{4\Gamma^2(H+\frac12)\sin^2\left( \frac{\pi (H-\frac12)}{2}\right) \cos^2\left( \frac{\pi (H-\frac12)}{2}\right) } \\
&=\frac{\pi^2}{\Gamma^2(H+\frac12)\sin^2\left( \pi (H-\frac12)\right) },
\end{split}
\end{equation*}
where in the second equation we have   used some standard formulas for improper integrals (see e.g. \cite[Page 331-332]{Z}).
Plugging this 
identity into the expression \eqref{eqn.sigma.var} we have obtained for
 $\Sigma(\theta)$, and using the definition of $C_H$ in (\ref{ch}) and of $c_{5}(H)$ in \eqref{eqn.c5},  we obtain that
\begin{equation*} \begin{split}
\Sigma(\theta)=\frac{\pi C_H}{\theta c_5(H)}\frac{\pi^2}{\Gamma^2(H+\frac12)\sin^2\left( \pi (H-\frac12)\right) }= \frac{
 \Gamma(2H+1)}{4 \theta H\Gamma(2H) }=\frac{1}{2\theta},
\end{split}
\end{equation*}
which concludes the result.
\end{proof}

\subsection{Efficiency of the MLE}

Consider now the MLE $\hat{\theta}_{\tau}$ of $\theta$ from the observation of a fractional Ornstein-Uhlenbeck process $Y$ in $\ott$, as defined in \cite{KL}.
It is well-known (see \cite{BCS, BK, KL}) that $\hat{\theta}_{\tau}$ is uniformly consistent on $\Theta$ (recall that $\Theta$ is assumed to be a compact subinterval of $(0, \infty)$ in the Ornstein-Uhlenbeck sense). That is, for any $\lambda>0$, we have
\begin{equation*} 
\lim_{\tau\to\infty} \sup_{\theta \in \Theta} {\bp}^{\tau}_{\theta}\left( \left\vert \hat{\theta}_{\tau}-\theta \right\vert>\lambda \right) = 0.
\end{equation*}
The estimator $\hat{\theta}_{\tau}$ is also uniformly asymptotically normal:
\begin{equation*} 
\mathcal{L}({\bp}_{\theta})-\lim_{\tau\to\infty} \sqrt{\tau} (\hat{\theta}_{\tau}-\theta))=\mathcal{N}(0,2\theta),
\end{equation*}
where the limit is uniform in $\theta \in \Theta$. Moreover, we have a uniform convergence of moments. Namely, for any $p>0$, one gets:
\begin{equation}\label{a0}
\lim_{\tau\to\infty} \sup_{\theta \in \Theta}\left\vert \be_{\theta} \lc \left\vert \sqrt{\tau}(\hat{\theta}_{\tau}-\theta)\right\vert^p \rc
-
\be_{\theta} \lc\left\vert \sqrt{2 \theta} Z\right\vert^p \rc \right\vert=0,
\end{equation}
where $\mathcal{L}(Z)=\mathcal{N}(0,1)$. 
In particular, the last two results   already suggest that the rate of convergence for the LAN property in this case is $\tau^{-1/2}$, as mentioned in \cite[p.162]{Rao}. Our Theorem \ref{lan} confirms this intuition, and one can see that the MLE reaches this optimal rate of order $\tau^{-1/2}$. 

Moreover, as a consequence of Theorem \ref{lan} and Proposition \ref{ousigma}, we obtain the asymptotic efficiency for polynomial loss functions  of the MLE for the fractional Ornstein-Uhlenbeck process, in   the sense of Theorem \ref{thm:minimax}.
\begin{proposition}
Let $Y$ be the fractional Ornstein-Uhlenbeck process solution to \textnormal{(\ref{ou})}. Then the MLE is asymptotically minimax efficient for any loss function $\ell(u)=\vert u \vert ^p$, $p>0$.
\end{proposition}

\begin{proof}
Fix $\theta \in \Theta$. Then, for all $\delta>0$ and $p>0$,
\begin{equation*}\begin{split}
&\sup_{\theta'\in \Theta:\vert\theta'-\theta\vert<\delta}\left\vert \be_{\theta'} \lc \left\vert \sqrt{\tau}(\hat{\theta}_{\tau}-\theta')\right\vert^p \rc
-
\be_{\theta} \lc\left\vert \sqrt{2 \theta} Z\right\vert^p \rc \right\vert\\
&\leq \sup_{\theta'\in \Theta}\left\vert \be_{\theta'} \lc \left\vert \sqrt{\tau}(\hat{\theta}_{\tau}-\theta')\right\vert^p \rc
-
\be_{\theta'} \lc\left\vert \sqrt{2 \theta'} Z\right\vert^p \rc \right\vert\\
&\qquad \qquad \qquad +
\sup_{\theta'\in \Theta:\vert\theta'-\theta\vert<\delta}\left\vert \be_{\theta'} \lc\left\vert \sqrt{2 \theta'} Z\right\vert^p \rc 
-
\be_{\theta} \lc\left\vert \sqrt{2 \theta} Z\right\vert^p \rc \right\vert,
\end{split}
\end{equation*}
where $\mathcal{L}(Z)=\mathcal{N}(0,1)$. 
Then, the uniform convergence of the moments (\ref{a0}) implies that 
\begin{equation*}\begin{split}
&\lim_{\tau \rightarrow \infty}\sup_{\theta'\in \Theta:\vert\theta'-\theta\vert<\delta}\left\vert \be_{\theta'} \lc \left\vert \sqrt{\tau}(\hat{\theta}_{\tau}-\theta')\right\vert^p \rc
-
\be_{\theta} \lc\left\vert \sqrt{2 \theta} Z\right\vert^p \rc \right\vert\\
&\quad \quad \leq \sup_{\theta'\in \Theta:\vert\theta'-\theta\vert<\delta}\left\vert \be_{\theta'} \lc\left\vert \sqrt{2 \theta'} Z\right\vert^p \rc 
-
\be_{\theta} \lc\left\vert \sqrt{2 \theta} Z\right\vert^p \rc \right\vert.
\end{split}
\end{equation*}
Therefore, 
\begin{equation*}
\lim_{\delta\to 0}\lim_{\tau \rightarrow \infty}\sup_{\theta'\in \Theta:\vert\theta'-\theta\vert<\delta}\left\vert \be_{\theta'} \lc \left\vert \sqrt{\tau}(\hat{\theta}_{\tau}-\theta')\right\vert^p \rc
-
\be_{\theta} \lc\left\vert \sqrt{2 \theta} Z\right\vert^p \rc \right\vert=0.
\end{equation*}
Thus, by Proposition \ref{ousigma},  the lower bound in \eqref{eq:low-bnd-minimax} is achieved by the MLE when $\ell(u)=\vert u \vert^p$, which completes the proof.
\end{proof}

\end{document}